\documentclass[11pt]{article}
\usepackage{enumerate}
\usepackage{verbatim}
\usepackage[english]{babel}
\usepackage[utf8]{inputenc}
\usepackage{listings}
\usepackage{authblk}
\usepackage{geometry}
\usepackage{psfrag}
\usepackage{epic}
\usepackage{lipsum}
\usepackage{eepic}
\usepackage{color}
\usepackage{graphicx}
\usepackage{tikz}
\usepackage{tkz-graph}
\usetikzlibrary{positioning,chains,fit,shapes,calc,arrows.meta}
\usepackage[caption=false]{subfig}
\usepackage{indentfirst}
\usepackage{float}
\usepackage{amsmath}
\usepackage{amssymb}
\usepackage{amsthm}
\usepackage{hyperref}

\usepackage[normalem]{ulem}
\usepackage{indentfirst}

\newtheorem{theo}{Theorem}[section]
\newtheorem{theorem}[theo]{Theorem}

\newtheorem{corollary}[theo]{Corollary}

\newtheorem{lemma}[theo]{Lemma}

\newtheorem{conjecture}[theo]{Conjecture}

\newtheorem{question}{Question}

\newcommand{\ini}{\ensuremath{\textrm{ini}}}
\newcommand{\chil}{\ensuremath{\vec{\chi}}}

\newcommand{\sP}{\mathcal{P}}
\newcommand{\sS}{\mathcal{S}}
\newcommand{\sQ}{\mathcal{Q}}
\newcommand{\sC}{\mathcal{C}}
\newcommand{\sH}{\mathcal{H}}
\newcommand{\sR}{\mathcal{R}}

\title{Orthogonality between acyclic subdigraphs and \\ paths in digraphs}

\author[1]{Caroline A. de Paula Silva\thanks{Supported by FAPESP (grant 2022/03735-0).}}

\author[2]{Cândida Nunes da Silva\thanks{Supported by FAPESP (grant 2023/03167-5).}}

\author[1]{Orlando Lee\thanks{Supported by FAPESP (grants 2023/03167-5 and 2015/11937-9) and CNPq (grant 302596/2022-4).}}

\affil[1]{\footnotesize{Institute of Computing, University of Campinas, Campinas, São Paulo, Brazil} \newline \{caroline.silva, lee\}@ic.unicamp.br}
\affil[2]{\footnotesize{Department of Computing, Federal University of São Carlos, Sorocaba, São Paulo, Brazil} candida@ufscar.br}

\begin{document} 

\maketitle

\begin{abstract}
Let $D$ be a digraph. A collection of disjoint sets of vertices (respec., collection of disjoint subdigraphs) $\sH$ of $D$ and a vertex subset (or subdigraph) $Q$ of $D$ are orthogonal if every set (respec., subdigraph) $H \in \sH$ contains exactly one vertex of $Q$.
A well-known result of Gallai and Milgram shows that for every minimum path partition of a digraph there is a stable set orthogonal to it. Similarly, Gallai, Hasse, Roy and Vitaver independently proved that for every longest path of a digraph there is a vertex partition into stable sets (i.e, vertex-coloring) orthogonal to it.
Berge showed that no analogous statements hold when optimality is required for the stable set or the vertex coloring. In this paper, we show that this holds if we replace stable sets by induced acyclic subdigraphs.

In 1981, Linial proposed two generalizations of Gallai-Milgram and Gallai-Hasse-Roy-Vitaver results using a positive integer $k$ as a measure of optimality for the path partition and the coloring, respectively. These generalizations have led to two conjectures that remain open. Using the same strategy of replacing stable sets by induced acyclic subdigraphs, we prove relaxations of both conjectures. 
\end{abstract}

\section{Introduction}
We assume that the reader is familiar with common terminology of graphs and digraphs~\cite{ bang2008digraphs, bomu08}.
Let $D=(V, A)$ denote a digraph. The \emph{underlying graph} of $D$, denoted by $U(D)$, is the simple graph with vertex set $V(D)$ such that $u$ and $v$ are adjacent in $U(D)$ if and only if $(u, v) \in A(D)$ or $(v, u) \in A(D)$. 
A \emph{stable set} of a $D$ is a stable set of $U(D)$. We denote the size of a maximum stable set of $D$ by $\alpha(D)$. 
A \emph{path partition} of $D$ is a collection of disjoint paths that covers $V(D)$. The size of a minimum path partition of $D$ is denoted by $\pi(D)$.
We denote by $\ini(P)$ the initial vertex of a path $P$. Similarly, if $\sP$ is a collection of paths, we denote by $\ini(\sP)$ the set of initial vertices of the paths in $\sP$.
A collection of disjoint sets of vertices (respec., collection of disjoint subdigraphs) $\sH$ of $D$ and a vertex subset (or subdigraph) $Q$ of $D$ are orthogonal if every set (respec., subdigraph) $H \in \sH$ contains exactly one vertex of $Q$. We also say that $\sH$ is \emph{orthogonal to $Q$} and vice versa.

In 1960, Gallai and Milgram~\cite{GallaiMilgram1960} showed that for every minimum path partition of $D$ there is a stable set orthogonal to it.

\begin{theorem}[Gallai-Milgram Theorem]
\label{th:gallai-milgram}
Let $D$ be a digraph.
For every minimum path partition $\sP$ of $D$, there exists a stable set $S$ of $D$ such that $\sP$ and $S$ are orthogonal.
In particular, $\pi(D) \leq \alpha(D)$.
\end{theorem}

One interesting question is if we require optimality on the stable set rather than the path partition, can we obtain a result similar to Theorem~\ref{th:gallai-milgram}? 

\begin{question}\label{question:alpha}
    Let $D$ be a digraph and let $T$ be a maximum stable set of $D$. Is it true that $D$ contains a path partition orthogonal to $T$?
\end{question}

Berge~\cite{Berge1982b} observed that the answer to this question is negative by exhibiting orientations of odd cycles that do not satisfy the desired property. See Figure~\ref{fig:anti-directed-c5}. Note that the set $\{v_3,v_5\}$ is a maximum stable set of the digraph $D$ depicted, but there is no orthogonal path partition to it.

\begin{figure}[htb]
    \centering
    \subfloat[\label{fig:anti-directed-c5}]{\includegraphics[height=5cm, width=5cm]{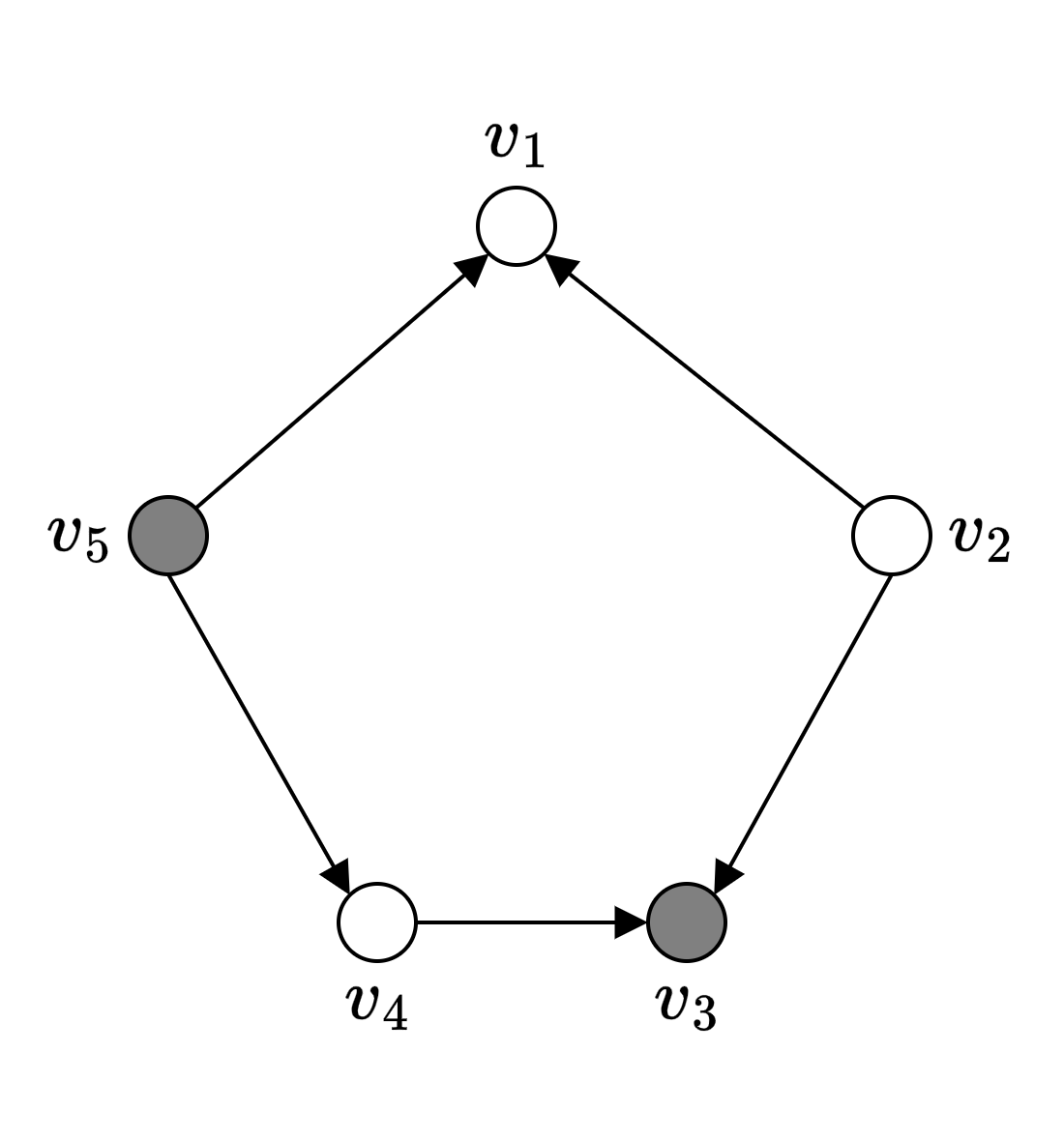}}
    \qquad
    \qquad
    \subfloat[\label{fig:conflicting-c5}]{\includegraphics[height=5cm, width=5cm]{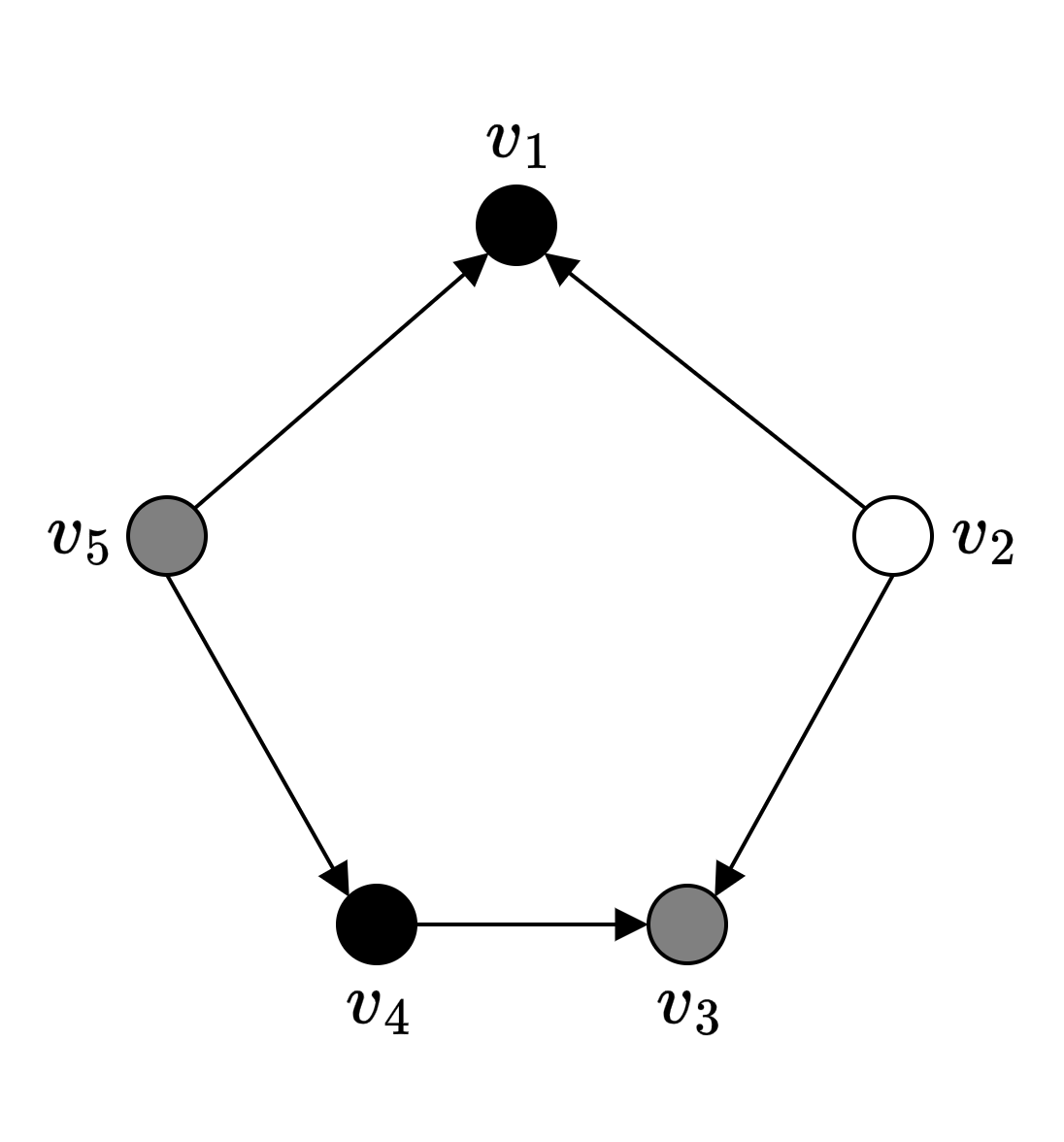}}
    \caption{Orientation $D$ of a $C_5$ that is a counterexample to both Question~\ref{question:alpha} and Question~\ref{question:chi}.}
    \label{fig:odd-cycles}
\end{figure} 

A \emph{coloring} of a digraph $D$ is a coloring of $U(D)$, that is, a partition of $V(D)$ into stable sets.
The \emph{chromatic number} of $D$, denoted by $\chi(D)$, is the chromatic number of $U(D)$. We denote by $\lambda(D)$ the order of a longest path of $D$.
We may obtain \emph{dual} results by exchanging the roles of stable sets and paths. In other words, we replace path partition by (vertex)-coloring and replace stable set by path. 

In the 1960s, a result analogous to the Gallai-Milgram Theorem was proved independently by Gallai~\cite{gallai1968directed}, Hasse~\cite{hasse1965algebraischen}, 
Roy~\cite{roy1967nombre} and 
Vitaver~\cite{vitaver1962determination}.

\begin{theorem}[Gallai-Hasse-Roy-Vitaver Theorem] \label{th:gallai-roy}
Let $D$ be a digraph. For every maximum path $P$ of $D$, there exists a coloring $S$ of $D$ such that $P$ and $S$ are orthogonal. In particular, $\chi(D) \leq \lambda(D)$.
\end{theorem}

Similarly as in the previous case, one interesting question is if we require optimality on the coloring rather than on the path, can we obtain a similar result? 

\begin{question}\label{question:chi}
    Let $D$ be a digraph and let $\sS$ be a minimum coloring of $D$. Is it true that $D$ contains a path $P$ orthogonal to $\sS$?
\end{question}

Berge~\cite{Berge1982b} observed that the answer to this question is also negative. See Figure~\ref{fig:conflicting-c5}. Note that $\{\{v_1, v_4\}, \{v_2\}, \{v_3, v_4\}\}$ is a minimum coloring of the digraph $D$ depicted, but there is no path orthogonal to it.

Motivated by the previous observations and by his study on perfect graphs, Berge~\cite{Berge1982b} introduced two classes of digraphs, named $\alpha$-diperfect and $\chi$-diperfect digraph. 
A digraph $D$ is \emph{$\alpha$-diperfect} if for every maximum stable set $S$ of $D$, there is a path partition orthogonal to $S$ and this property holds for every induced subdigraph of $D$.
A digraph $D$ is \emph{$\chi$-diperfect} if for every minimum coloring $\sS$ of $D$, there is a path $P$ orthogonal to $\sS$, and this property holds for every induced subdigraph of $D$. 
The ultimate goal for the classes of $\alpha$-diperfect digraph and $\chi$-diperfect digraphs is to obtain a characterization of these classes in terms of forbidden induced subdigraphs.
Some advance has been made towards this goal~\cite{chi-diperfect-2022, chi-diperfect-stab-two-2022, freitas2022-3anti, freitas2022some, sambinelli2022alpha, de2025acyclic}, and it seems that this is a very hard problem. Many examples of minimal non-$\alpha$-diperfect digraphs and minimal non-$\chi$-diperfect digraphs have been discovered and they do not seem to have an evident pattern~\cite{chi-diperfect-obstructions, de2023family, de2024diperfect}. 

In 1981, Linial~\cite{linial1981extending} proposed two generalizations of Gallai-Milgram Theorem and Gallai-Hasse-Roy-Vitaver Theorem that use a positive integer $k$ as a measure of optimality for the path partition and the coloring, respectively. These generalizations have led to two conjectures that remain open~\cite{berge1982k, hartman2006berge, tesemaycon2018}.  
Let $D$ be a digraph and let $k$ be a positive integer. The \emph{$k$-norm} of a path partition $\sP$ of $D$, denoted by $|\sP|_k$, is defined as $\sum_{P \in \sP} \min\{|V(P)|, k\}$. A path partition is \emph{$k$-minimum} if $|\sP|_k \leq |\sP'|_k$ for every path partition $\sP'$ of $D$. The $k$-norm of a $k$-minimum path partition is denoted by $\pi_k(D)$. A \emph{partial $k$-coloring} $\sS = \{S_1, \ldots, S_k\}$ is a collection of $k$ vertex-disjoint stable sets. The maximum number of vertices that can be covered by a partial $k$-coloring of $D$ is denoted by $\alpha_k(D)$. Observe that, for $k = 1$, a minimum $1$-path partition is simply a minimum path partition and a partial $1$-coloring is a stable set. By Gallai-Milgram Theorem, we have $\alpha_1(D) \geq \pi_1(D)$.
Linial~\cite{linial1981extending} conjectured that this inequality holds for every positive integer $k$.

\begin{conjecture}[Linial]\label{conj:linial-primal}
    For every digraph $D$ and every positive integer $k$, it follows that $\pi_k(D) \leq \alpha_k(D)$.
\end{conjecture}

The \emph{$k$-norm} of a coloring $\sS$ of $D$, denoted by $|\sS|_k$, is defined as $\sum_{S \in \sS} \min\{|S|, k\}$. A coloring is \emph{$k$-minimum} if $|\sS|_k \leq |\sS'|_k$ for every coloring $\sS'$ of $D$. The $k$-norm of a $k$-minimum coloring is denoted by $\chi_k(D)$.
A \emph{$k$-pack} $\sP = \{P_1, \ldots, P_k\}$ is a collection of $k$ disjoint paths. The maximum number of vertices that can be covered by a $k$-path of $D$ is denoted by $\lambda_k(D)$. Observe that, for $k = 1$, a $1$-minimum coloring is simply a minimum coloring and a $1$-pack is simply a path. By Gallai-Hasse-Roy-Vitaver Theorem, we have $\lambda_1(D) \geq \chi_1(D)$. Similarly as before, Linial~\cite{linial1981extending} conjectured that this inequality holds for every positive integer $k$. 

\begin{conjecture}[Linial]\label{conj:linial-dual}
    For every digraph $D$ and every positive integer $k$, it follows that $\chi_k(D) \leq \lambda_k(D)$.
\end{conjecture}

In 1982, Neumann-Lara~\cite{neumann1982dichromatic} introduced a directed analogue of the usual concept of coloring in a digraph. In this scenario, stable sets are replaced by (induced) acyclic digraphs. 
A digraph is \emph{acyclic} if it does not contain any directed cycles.
Let $D$ be a digraph. 
A \emph{dicoloring} of a digraph is a partition $S_1, \ldots, S_t$ of $V(D)$ such that, for every $i \in \{1, \ldots, t$\}, $D[S_i]$ is an acyclic digraph. The minimum $t$ for which $D$ admits a dicoloring of size $t$ is called the \emph{dichromatic number} of $D$ and it is denoted by $\chil(D)$. In the early 2000s, this notion was reintroduced by Mohar et al.~\cite{mohar2003circular, bokal2004circular} and since then many results analogous to classical results in graph theory have been proven using the notion of dicoloring~\cite{aboulker2021extension, picasarri2024strengthening, kawarabayashi2025analogue}. We refer the reader the PhD thesis of Aubian~\cite{aubian2023colouring} and Picassari-Arrieta~\cite{picasarri2024digraph} for a more complete overview of the topic.

Since every stable set corresponds to an acyclic digraph, Theorems~\ref{th:gallai-milgram} and~\ref{th:gallai-roy} are trivially true if we replace stable sets by induced acyclic subdigraph and, coloring by dicoloring, respectively. 
It is natural to ask, then, what happens if we make the same change in Questions~\ref{question:alpha} and~\ref{question:chi}.

\begin{question}\label{question:new-alpha}
    Let $D$ be a digraph and let $T$ be a maximum induced acyclic subdigraph of $D$. Is it true that $D$ contains a path partition orthogonal to $T$?
\end{question}

\begin{question}\label{question:new-chi}
    Let $D$ be a digraph and let $\sS$ be a minimum dicoloring of $D$. Is it true that $D$ contains a path $P$ orthogonal to $\sS$?
\end{question}

It turns out that this change makes both problems much simpler.
In this paper, we give affirmative answers for both questions.

We may define analogous versions of  partial $k$-colorings and $k$-minimum colorings by replacing stable sets by induced acyclic digraphs. 
Let $D$ be a digraph and let $k$ be a positive integer.
A \emph{partial $k$-dicoloring} $\sS = \{S_1, \ldots, S_k\}$ of $D$ is a collection of $k$ disjoint induced acyclic subdigraphs of $D$. The maximum number of vertices that can be covered by a partial $k$-dicoloring of $D$ is denoted by $\vec{\alpha}_k(D)$. 
Let $\sS$ be a dicoloring of $D$. The \emph{$k$-norm} of $\sS$, denoted by $|\sS|_k$, is defined as $\sum_{S \in \sS} \min\{|S|, k\}$. A dicoloring is \emph{$k$-minimum} if $|\sS|_k \leq |\sS'|_k$ for every dicoloring $\sS'$ of $D$. The $k$-norm of a $k$-minimum dicoloring of $D$ is denoted by $\vec{\chi}_k(D)$. Since every stable set induces an acyclic digraph, $\vec{\alpha}_k(D) \geq \alpha_k(D)$ and $\chi_k(D) \geq \vec{\chi_k}(D)$. 
We prove the following relaxations of Conjectures~\ref{conj:linial-primal} and ~\ref{conj:linial-dual}. 

\begin{theorem}\label{th:relax-linial}
    For every digraph $D$ and every positive integer $k$, $\pi_k(D) \leq \vec{\alpha}_k(D)$ and $\vec{\chi}_k(D) \leq \lambda_k(D)$.
\end{theorem}

\section{Answering Questions~\ref{question:new-alpha} and~\ref{question:new-chi}}

Before we present our results, we need some auxiliary concepts.
Let $G$ be a graph, let $X \subseteq V(G)$, and let $M$ be a matching of $G$. We say that $M$ \emph{covers} $X$ if every vertex of $X$ is incident to some edge of $M$.
If $S \subseteq V(G)$, we denote by $N_G(S)$ the set of neighbors of $S$ in $G$. When $G$ is clear from the context, we write simply $N(S)$.
The following theorem is a classical result in Graph Theory, due to Hall~\cite{hall1987representatives}.

\begin{theorem}[Hall's Theorem]\label{th:hall}
    A bipartite graph $G[X, Y]$ has a matching that covers $X$ if and only if
$|N(S)| \geq |S|$ for all $S \subseteq X$.
\end{theorem}

If $D$ is a digraph and $X_1, X_2 \subseteq V(D)$ are disjoint sets of $D$, then we denote by $D[X_1, X_2]$ the bipartite subdigraph of $D$ with vertex-set $X_1 \cup X_2$ and arc-set containing all arcs of $D$ between $X_1$ and $X_2$.
A \emph{matching} of a digraph is a matching of its underlying graph.
We say that $M \subseteq A(D)$ is a \emph{directed matching} in $D[X_1, X_2]$ if $M$ is a matching in $U(D[X_1, X_2])$ and all arcs of $M$ are directed from $X_1$ to $X_2$. Moreover, we say that $M$ \emph{covers} $X_1$ (respec., \emph{covers} $X_2$) if $M$ covers $X_1$ (respec., covers $X_2$) in $U(D[X_1, X_2])$. 

\begin{lemma}\label{lem:directed-matching}
    Let $D$ be a digraph and let $T_1$ be a maximum induced acyclic subdigraph of $D$. If $T_2$ is an induced acyclic subdigraph of $D - V(T_1)$, then there exists a directed matching covering $V(T_2)$ in $D[V(T_1), V(T_2)]$. 
\end{lemma}
\begin{proof}
    Let $X_1 = V(T_1)$ and let $X_2 = V(T_2)$.
    Let $D'$ be the subdigraph of $D$ with vertex set $X_1\cup X_2$ and containing only the arcs of $D$ that are directed from $X_1$ to $X_2$. 
    Let $G = G[X_1, X_2] = U(D')$. Note that if $G[X_1, X_2]$ has a matching that covers $X_2$, then there is a directed matching in $D[X_1, X_2]$ that covers $X_2$ and the result follows. Towards a contradiction, assume that such a matching does not exist. Then, by Hall's Theorem, there exists a subset $S \subseteq X_2$ such that $|S| > |N_G(S)|$. Also note that $S\cup (X_1\setminus N_G(S))$ induces an acyclic subdigraph of $D$. However, $|S\cup (X_1\setminus N_G(S))|  > |X_1| = |V(T_1)|$, a contradiction.
\end{proof}

A \emph{greedy dicoloring} $\sS = \{S_1, \ldots, S_t\}$ of $D$ is a dicoloring such that $S_1$ is a maximum induced acyclic subdigraph of $D$ and $S_i$ is a maximum induced acyclic subdigraph of $D - \bigcup_{j = 1}^{i-1} S_j$, for every $i \in \{2, \ldots, t\}$. A \emph{good path partition} with respect to a greedy dicoloring $\sS = \{S_1, \ldots, S_t\}$ of $D$ is a path partition $\sP$ of $D$ such that for every path $(v_1, \ldots, v_\ell) \in \sP$ and every $i \in \{1, \ldots, \ell\}$, it follows that $v_i \in S_i$.

\begin{lemma}\label{lem:good-path-part}
    Let $D$ be a digraph and let $\sS=\{S_1, \ldots, S_t\}$ be a greedy dicoloring of $D$. Then there is a good path partition with respect to $\sS$.
\end{lemma}
\begin{proof}
    The proof follows by induction on $|V(D)|$. If $t = 1$, then the path partition consisting of $|V(D)|$ trivial paths satisfies the desired property. So we may assume that $t \geq 2$. By induction hypothesis, there exists a good path partition $\sP'$ of $D - S_1$ with respect to $\{S_1'= S_2, \ldots, S'_{t-1}= S_t\}$. Note that, by definition of a good path partition, for every $P \in \sP'$, $\ini(P) \in S'_1 = S_2$.
    
    By Lemma~\ref{lem:directed-matching}, there exists a directed matching $M$ covering $S_2$ in $D[S_1, S_2]$. Let $\sQ$ be the set of trivial paths consisting of every vertex of $S_1$ that is not covered by $M$. Let $\sR$ be the set of disjoint paths obtained from $\sP'$ by extending each path $P \in \sP'$ with the arc of $M$ incident to $\ini(P)$. Then $\sP = \sR \cup \sQ$ is a good path partition of $D$ with respect to $\sS$. This finishes the proof.
\end{proof}

Lemma~\ref{lem:good-path-part} immediately gives an affirmative answer to Question~\ref{question:new-alpha}.

\begin{corollary}\label{th:yes-alpha}
    Let $D$ be a digraph and let $T$ be a maximum induced acyclic subdigraph of $D$. There exists a path partition $\sP$ of $D$ orthogonal to $T$. 
\end{corollary}

Now we prove Theorem~\ref{th:yes-chi}
which gives an affirmative answer to Question~\ref{question:new-chi}.

\begin{theorem}\label{th:yes-chi}
    Let $D$ be a digraph and let $\sS$ be a minimum dicoloring of $D$. Then there exists a path $P$ of $D$ orthogonal to $\sS$.
\end{theorem}
\begin{proof} 
Let $t = \chil(D)$ and let $\sS=\{S_1, \ldots, S_t\}$ be a minimum dicoloring of $D$. Let $F=\{(u,v)\in A(D): u\in S_i, v\in S_j, i<j\}$.  Note that since each $D[S_i]$ is acyclic, $D-F$ is also acyclic. Consider the subdigraph $D'=(V(D),F)$.
Suppose first that $D'$ has a path $P$ of order $t$. Then clearly $P$ is a path orthogonal to $\sS$. So we may assume that the longest path of $D'$ has order at most $t - 1$. By Gallai-Hasse-Roy-Vitaver Theorem, $\chi(D') \leq t - 1$. Let $\sC=\{C_1, \ldots, C_{\chi(D')}\}$ be a minimum coloring of $D'$. 
Since each color class $C_j$ is a stable set in $D'$, the arc set of $D[C_j]$ contains no arc of $F$ and, so $D[C_j]$ is a subdigraph of $D-F$.
This implies that $D[C_j]$ is acyclic. Hence, $\{C_1, \ldots, C_{\chi(D')}\}$ is a dicoloring of $D$ of size smaller than $t$, a contradiction.
\end{proof}

\section{Linial's Conjectures}
\label{ap:linial}

In this section, we use the ideas of greedy dicoloring and good path partition to prove Theorem~\ref{th:relax-linial}

\begin{proof}[Proof of Theorem~\ref{th:relax-linial}]
    Let $\sS = \{S_1, \ldots, S_t\}$ be a greedy dicoloring of $D$ and let $\sP$ be a good path partition of $D$ with respect to $\sS$ (it exists by Lemma~\ref{lem:good-path-part}). 

    Let us first show that $\vec{\alpha}_k(D) \geq \pi_k(D)$.
    Note that if $t \leq k$, then $\vec{\alpha}_k(D) = |V(D)|$ and the result is trivial. So we may assume that $t > k$.
    Observe that $\vec{\alpha}_k(D) \geq |\bigcup_{i=1}^k S_i|$.
    Let $\sP_1$ be the set of paths of $\sP$ of order at most $k - 1$ and let $\sP_2$ be the set of paths of $\sP$ of order at least $k$. Note that, by definition of a good path partition, 
    every vertex of $S_k, \ldots, S_t$ is covered by $\sP_2$. Moreover, exactly $|S_k|$ vertices of each $S_i$ with $i \in \{1, \ldots, k-1\}$ are covered by $\sP_2$. This implies that $|\sP_2|=|S_k|$ and $\sP_1$ covers exactly $|\bigcup_{i=1}^{k-1} S_i| - (k-1)|S_k|$ vertices. So the $k$-norm of $\sP$ is $|V(\sP_1)| + k|\sP_2| = |\bigcup_{i=1}^{k-1} S_i| - (k-1)|S_k| + k|S_k| = |\bigcup_{i=1}^{k} S_i|$. Hence,
    \[ \pi_k(D) \leq |\sP|_k = |\bigcup_{i=1}^{k} S_i| \leq \vec{\alpha}_k(D).  \]

    Let us now show that $\lambda_k(D) \geq \vec{\chi}_k(D)$. If $|\sP|\leq k$, then $\lambda_k(D)=|V(D)|$ and the result is clear. So we may assume that $|\sP| > k$.
        Let $\sP'$ be a set of  $k$ longest paths of $\sP$. By definition of a good path partition, every $S_i$ such that $|S_i| \leq k$ is covered by $\sP'$. Since $\sS$ is a greedy dicoloring, $|S_{i-1}| \ge |S_i|$ for $i\in \{2,\ldots t\}$. If for every $i \in \{1,\ldots,t\}$, $|S_i| \leq k$, then $|\sP|\leq k$ which contradicts our assumption. So, let $j \in \{1,\ldots,k\}$ be the maximum integer such that $|S_j| > k$. Hence, $\lambda_k(D) \geq |V(\sP')| =  kj + |\bigcup_{i = {j+1}}^t S_i|$.
    Now the $k$-norm of $\sS$ is $kj + |\bigcup_{i = {j+1}}^t S_i|$. Hence,

    \[
    \vec{\chi}_k(D) \leq |\sS|_k = kj + |\bigcup_{i = {j+1}}^t S_i| \leq \lambda_k(D).
    \]

    This finishes the proof.
\end{proof}

\bibliographystyle{abbrv} 
\bibliography{bibliography} 

\end{document}